\documentclass[12pt,reqno]{amsart}
\usepackage{verbatim,amssymb,amsmath,amscd,latexsym,amsbsy,mathrsfs}
\usepackage[inline]{enumitem}   
\usepackage{amsthm}
\usepackage{mathtools}
\usepackage[all,cmtip]{xy}
\DeclareMathOperator*{\Hom}{Hom}
 
\DeclareMathOperator*{\Aut}{Aut}

\DeclareMathOperator*{\Spec}{Spec}
\DeclareMathOperator*{\ram}{ram}

\DeclareMathOperator*{\Pic}{Pic}
\DeclareMathOperator*{\P0}{Pic^0}

\DeclareMathOperator*{\e}{\acute{e}t}
\DeclareMathOperator*{\id}{Id}
\DeclareMathOperator*{\dv}{div}

\DeclareMathOperator*{\into}{\hookrightarrow}
\newcommand{\chash}{\mathcal{\#}} 
\newcommand{\Z}{\mathbb{Z}} \newcommand{\Aff}{\mathbb{A}} \newcommand{\PP}{\mathbb{P}}
\newcommand{\FF}{\mathbb{F}}
\newcommand{\cO}{\mathcal{O}}
\newcommand{\onto}{\relbar\joinrel\twoheadrightarrow}

\newtheorem{thm}{Theorem}[section]

\numberwithin{equation}{section}

\newtheorem{pro}[thm]{Proposition}
\newtheorem{lm}[thm]{Lemma}
\newtheorem{co}[thm]{Corollary}
\theoremstyle{remark} 
\theoremstyle{definition}
\newtheorem{rmk}[thm]{Remark}

\newtheorem{notation}[thm]{Notation}
\newtheorem{ex}[thm]{Example}

\newtheorem{cnd}[thm]{Condition}

\title{
The minimum genus of Galois covers of curves}
\author{Manish Kumar} \author{Poulami Mandal}
\date{}

\begin{document}
\begin{abstract}
Let $V\longrightarrow X$ be a $G$-Galois connected cover of smooth projective connected curves over an algebraically closed field $k$ of positive characteristic $p$ with the branch locus contained in a finite subset of closed points $S_X$ in $X$, where $G$ is a finite cyclic $p$-group. Let $l$ be a prime number other than $p$. Let $\Gamma$ be an extension of an elementary abelian $l$-group $H$ by $G$. We find $G$-stable submodules of the $l$-torsion of the Picard group of $V$ and its generalisation.
This is used to describe a method for finding 
the minimum genus of $\Gamma$-covers of $X$, \'etale over $X\setminus S_X$, dominating $V$; and also the minimum of the genera of $\Gamma$-covers of $\PP^1$ \'etale over $\Aff^1$.
\end{abstract}
\maketitle

\section{Introduction}
Let $X$ be a smooth connected projective curve over an algebraically closed field $k$ of characteristic $p>0$, and $S_X$ be a finite subset of closed points in $X$. Since the \'etale fundamental group $\pi_1^{\e}(X\setminus S_X)$ is a profinite group, understanding its finite quotients is helpful in understanding its structure. Abhyankar's conjecture on affine curves, proved by Raynaud and Harbater (\cite[Th\'eor\`eme 2.2.1 ]{R}, \cite[Theorem~6.2]{Har} and \cite[Corollary~4.7]{Har2}) states that any finite group is a quotient of 
$\pi_1^{\e}(X\setminus S_X)$ if and only if the maximal prime-to-$p$ quotient of this group has a generating set of cardinality less than $2g_X+r_X-1$, where $g_X$ is the genus of $X$ and $r_X>0$ is the cardinality of $S_X$. However, this does not completely describe the profinite group structure of $\pi_1^{\e}(X\setminus S_X)$. Given two finite quotient groups $\Gamma$ and $G$ of $\pi_1^{\e}(X\setminus S_X)$ where $G$ is a quotient of $\Gamma$, we want to know how they fit in the inverse system of finite quotients of $\pi_1^{\e}(X\setminus S_X)$. In other words, given a $G$-cover $\psi:V\longrightarrow X$ \'etale away from $S_X$, we are interested in $\Gamma$-covers $W\longrightarrow X$ \'etale away from $S_X$ that dominate $V\longrightarrow X$. Note that $S_X$ contains the branch locus of $\psi$, but it may also contain some non-branch points. 


The existence and non-existence of such $\Gamma$-covers have been studied in \cite{Pop}, \cite{HS}, \cite{Kum}. Also see survey papers \cite{PS} and \cite{HOPS} for more details.

A finite group is called quasi-$p$ if it is generated by its Sylow-$p$ subgroups. Let $H$ be the subgroup of $\Gamma$ such that $\Gamma/H\cong G$. When $S_X$ is non-empty and $H$ is quasi-$p$, such $\Gamma$-covers always exist (see \cite[Theorem B]{Pop}, \cite[Corollary~4.6]{Har2}). If $S_X$ is empty and $H$ is a $p$-group, the existence of these $\Gamma$-covers was discussed by Pacheco, Stevenson (\cite[Theorem~1.3]{PaS}) and also Borne (\cite[Theorem~1.1]{Borne}). Hence, we investigate the situation when $H$ is prime-to-$p$ (abelian).

Let $l$ be a prime number different from $p$. Let $G$ be the cyclic group $\Z/p^a\Z$ and $H\cong (\Z/l\Z)^n$, for (fixed) natural numbers $a$ and $n$. Then $\Gamma\cong H\rtimes_\theta G$ (by the Schur-Zassenhaus Theorem) for some $G$-action $\theta:G\to\Aut(H)$.
It was shown in \cite{EPCh} that an $l$-torsion finite abelian group $P_l(V\setminus S_V)$, where $S_V$ is the inverse image of $S_X$ in $V$, can be useful for constructing an $H$-cover $W\longrightarrow V$, which is a $\Gamma$-cover of $X$ and dominates $V$. This group $P_l(V\setminus S_V)$ is a generalisation of the $l$-torsion part of the degree zero Picard group of $V$, $\P0(V)[l]$, and the Galois group $\Aut(V|X)\cong G$ acts on it from the right. Its linearly independent subsets of cardinality $n$ are in bijection with $(\Z/l\Z)^n$-covers of $V$ \'etale over $V\setminus S_V$ (Lemma~\ref{C}). It was shown in \cite{EPCh} that the $\Gamma$-covers of $X$, dominating $V$, for various group actions $\theta$ are in bijection with the $G$-submodules of $P_l(V\setminus S_V)$. Here we study the $G$-module structure of $P_l(V\setminus S_V)$(see Lemma~\ref{subcover_2}, Theorem~\ref{2lem} and Corollary~\ref{dim-ker}). It is not clear how $\P0(V)[l]$ (a direct summand of $P_l(V\setminus S_V)$) decomposes as a direct sum of irreducible $\Z/p^a\Z$-representations. 

Another way to understand the \'etale fundamental group is to examine the ramification groups at the branch points of Galois \'etale covers. One interesting question is to find Galois \'etale covers with prescribed ``small" degree of the ramification divisor. Since the genus of this cover is related to the ramification filtration, finding a Galois cover with the least genus is of interest. In \cite[Theorem~4.6]{MP} and \cite[Corollary~2.2]{B2}, some Galois \'etale covers of $\Aff^1$ with Galois group isomorphic to alternating groups were constructed, and it was shown that each of these covers was of minimum genus. In \cite{G}, for a prime number $l$ coprime to $p$, $(\Z/l\Z)^n\rtimes\Z/p\Z$-Galois covers of $\Aff^1_k$  with minimal genus were described. Note that the group $(\Z/l\Z)^n\rtimes\Z/p\Z$ depends on the action of $\Z/p\Z$ on $(\Z/l\Z)^n$. And in~\cite{G}, the minimal genus is computed among all such Galois covers with all possible actions of $\Z/p\Z$ on $(\Z/l\Z)^n$ (with even $n$ allowed to vary).

\subsection*{Main results}
Given a smooth projective connected curve $X$ of genus $g_X$ over an algebraically closed field $k$ of characteristic $p$, a finite set of closed points $S_X$ in $X$, a group action $\theta:\Z/p^a\Z\longrightarrow \Aut(\Z/l^n\Z)$, and a $(\Z/p^a\Z)$-cover $V\longrightarrow X$ with the branch locus contained in $S_X$, we discuss the minimum genus for an $(\Z/l\Z)^n$-cover $W\longrightarrow V$ such that $W\longrightarrow X$ is a $((\Z/l\Z)^n\rtimes_{\theta}(\Z/p^a\Z))$-Galois cover, \'etale away from $S_X$. More precisely, when $a$, $n$ and $\theta$ are fixed, we give a formula for the minimum genus possible for $W$ when (1) every branch point of $V \longrightarrow X$ is totally ramified or (2) the fibre of $V\longrightarrow X$ over each branch point has exactly $p$ points (see Theorem~\ref{min_gen}). The general case is described below:

Let $d_s$ be the order of $l$ in $(\Z/p^s\Z)^*$ for $1\le s\le a$. We label the $d_s$ dimensional irreducible $(\Z/p^a\Z)$-modules over $\FF_l$ by $M_{s,t}$, for $1\le t\le p^{s-1}(p-1)/d_s$ (see Notation~\ref{basic_not}). Let $a$, $n$ and $\theta$ be fixed. Let there be a $(\Z/p^a\Z)$-module monomorphism from $(\Z/l\Z)^n$ into $P_l(V\setminus S_V)$. Then, by \cite[Theorem 3.5]{EPCh}, there exists at least one $((\Z/l\Z)^n\rtimes_\theta(\Z/p^a\Z))$-cover $W\longrightarrow X$, dominating $V\longrightarrow X$.
Let $H'$ be a maximal $(\Z/p^a\Z)$-submodule of $(\Z/l\Z)^n$ that can be embedded in $\P0(V)[l]$ (a direct summand of $P_l(V\setminus S_V)$). Let the multiplicity of the $(\Z/p^a\Z)$-module $\FF_l$ with trivial $(\Z/p^a\Z)$-action, (resp. $M_{s,t}$) in $H/H'$ be $\alpha_0$ (resp. $\alpha_{s,t}$). Let $r_s$ be the number of points in $S_X$ for which each point has exactly $p^s$ pre-images in $V$.
\begin{thm}\label{min_alg}
A method of finding the genus $\mathfrak{g}$ of a $(\Z/l\Z)^n\rtimes_\theta (\Z/p^a\Z)$-cover of $X$, minimum among the genera of all $(\Z/l\Z)^n\rtimes_\theta (\Z/p^a\Z)$-covers of $X$, \'etale over $X\setminus S_X$, dominating $V\longrightarrow X$, is given by the following integer linear programming problem with bounded variables.

Minimize $$\sum_{s=0}^a p^sy_s$$
subject to : $\sum_{i=s}^a y_i\ge \max_t\{\alpha_{s,t}\},$ for $s>0$, $\sum_{s=0}^ay_s\ge \alpha_0+1$ and $0\le y_s\le r_s$ and $y_s\in\Z$, $0\le s\le a$.

If $c_{\min}$ is the minimum value of $\sum_{s=0}^a p^sy_s$, then the minimum genus is $$\mathfrak{g}=\frac{2+l^n(2g_V-2)+l^{n-1}(l-1)c_{\min}}{2}.$$
\end{thm}

Finally, in the Section~\ref{last_section}, we denote by $(\Z/l\Z)^n \rtimes (\Z/p^a\Z)$ the class of groups consisting of extensions of $(\Z/p^a\Z)$ by $(\Z/l\Z)^n$, where the $(\Z/p^a\Z)$-action $\theta$ on $\Z/l^n\Z$ is allowed to vary. For $a=1$, we find the minimum genus of $(\Z/l\Z)^n \rtimes \Z/p\Z$-covers of $\PP^1$ \'etale over $\Aff^1$ in Proposition~\ref{p-cyclic}. This recovers \cite[Theorem~4.2 and 4.3]{G} (see Remark~\ref{rmk-p-cyclic}). Moreover, Theorem~\ref{p^a}, Corollary~\ref{p2} and Theorem~\ref{p^2-Galois} extend the earlier results to obtain the minimum genus of $(\Z/l\Z)^n \rtimes \Z/p^a\Z$-covers. The minimum genus is computed among all Galois covers with all possible groups $(\Z/l\Z)^n \rtimes \Z/p^a\Z$ as the action of $\Z/p^a\Z$ on $(\Z/l\Z)^n$ varies.

\begin{thm}\label{p^2-Galois}
Let $a=2$ and $n$ be a positive integer. Let the order of $l$ in $(\Z/p\Z)^*$ and $(\Z/p^2\Z)^*$ be $d_1$ and $d_2$ respectively. The minimum genus of all $((\Z/l\Z)^n\rtimes (\Z/p^2\Z))$-Galois covers of $\PP^1$, \'etale over $\Aff^1$, for all positive integer $n$ and all group homomorphisms $\Z/p^2\Z\longrightarrow \Aut((\Z/l\Z)^n)$ is $$1+l^{d_1}\left(\frac{-1-p^2+(p-1)c}{2}\right)$$ where
\[c=
\begin{cases}
p^2+1, &\text{ if } d_1=d_2;\\
15, &\text{ if } d_1<d_2 \text{ and }p=2; \\
2p^2+2, &\text{ if } d_1<d_2 \text{ and }p\neq 2.
\end{cases}
\]
\end{thm}
Here the positive integer $n$ and the group action $\Z/p^2\Z\longrightarrow \Aut((\Z/l\Z)^n)$ are varying. Hence, the minimum is over a class of groups of the type $(\Z/l\Z)^n\rtimes (\Z/p^2\Z)$. Also note that in view of Lemma \ref{quasi-p} and Abhyankar’s conjecture (\cite{R}), there exists an \'etale $(\Z/l\Z)^n\rtimes (\Z/p^2\Z)$-cover of the affine line iff $d_1$ divides $n$.

We will adopt the following convention.
\begin{enumerate} 
 \item Unless otherwise mentioned, a cover will mean a connected cover. 
 \item A finite subset $B$ of a free $\Z$-module will be called linearly independent if the intersection of submodules generated by any two disjoint subsets of $B$ is zero. 
 \item We say that the covers $W_1\longrightarrow Y$, $\ldots$ ,$W_n\longrightarrow Y$ are mutually linearly disjoint if the fibre product $W_1\times_Y\ldots\times_YW_n$ is an irreducible and reduced scheme. \end{enumerate}

\subsection*{Acknowledgement}
We thank the referee for several useful suggestions, which significantly improved the paper.

\section{Cyclic covers of a curve and formulae of genus}\label{secn2}
Let $Y$ be a smooth connected projective curve of genus $g$ over an algebraically closed field $k$ of characteristic $p>0$, and $S$ be a finite subset of closed points in $Y$, with $|S|=r$. Let $m>1$ be an integer coprime to $p$, $\mathbf{\mu_m}=\Spec \Z[T] / (T^m-1)$. Let $\Pic(Y)$ denote the Picard group of $Y$ 
and $\Z[S]$ denote the Cartier divisors of $Y$ with support in $S$. Note that $\Z[S]$ is a free abelian group with rank $r$. Let $(\Z/m\Z)[S]=\Z[S]\otimes_{\Z}(\Z/m\Z)$ be the free module over $\Z/m\Z$ with basis $S$.
Let us define an $m$-torsion abelian group \begin{equation}
P_m(Y\setminus S):=\frac{\{([L],D)\in \Pic(Y)\oplus\mathbb{Z}[S]\mid L^{\otimes m}\cong\mathcal{O}(-D)\}}{\{(\mathcal{O}(-D),mD)\mid D\in\mathbb{Z}[S]\}}. 
\end{equation}
Then we have the following isomorphism of groups (see \cite[Proposition 3.5]{T}) $$P_m(Y\setminus S)\cong H^1_{\e}(Y\setminus S, \mathbf{\mu_m})\cong\Hom(\pi_1^{\e}(Y\setminus S),\Z/m\Z)\cong(\Z/m\Z)^{ 2g+r-1+b^{(2)}},$$
where $b^{(2)}=1$ if $r=0$, else $b^{(2)}=0$. This gives a bijection between $m$-cyclic \'etale covers of $Y\setminus S$ and the elements $([L],D)\in P_m(Y\setminus S)$ (for simplicity, we also denote the equivalence class of $([L],D)$ by $([L],D)$).

\subsection{Construction of $m$-cyclic covers}
Choose $([L],D)\in P_m(Y\setminus S)$. We may assume $D$ is an effective Cartier divisor with support in $S$, since the multiplicity at each point in $S$ can be chosen from $\{0,1, \ldots, m-1\}$. For any $([L],D)\in P_m(Y\setminus S)$, we have $L^{\otimes m}\otimes\cO(D)=\dv(f)\sim 0$, for a rational function $f\in k(Y)$. The corresponding $m$-cyclic cover (possibly disconnected) is the normalization of $Y$ in $k(Y)[t]/(t^m-f)$. 
If the cover is connected and $D=\sum_{i=1}^r a_i y_i$ for $a_i\in \{0,1, \ldots, m-1\}$ and $y_i\in S$, $1\le i\le r$, then this cover is tamely ramified at $y_i$ if and only if $\gcd(m,a_i)<m$. In this situation, the \textbf{ramification index} is $\frac{m}{\gcd(m,a_i)}$ (see \cite[3.15. Lemma]{EV}). Note that $[L]\mapsto ([L],0)$ gives an injective homomorphism from $\P0(Y)[m]$ to $P_m(Y\setminus S)$ and the corresponding cover will be \'etale over $Y$. Conversely an $m$-cyclic cover of $Y$, \'etale over $Y\setminus S$, corresponds to an element $([L],D)$ in $P_m(Y\setminus S)$ (see the discussion above \cite[Proposition~3.5]{T} 

The following well-known result describes important properties of $(\Z/m\Z)^n$-covers of $Y$, \'etale away from $S$.

\begin{lm}[{{\cite[Proposition 3.3]{EPCh}}}]\label{C}
For the curve $Y$ and the $\Z/m\Z$-module $P_m(Y\setminus S)$ as above, let $W_{\mathfrak{L}}\longrightarrow Y$ be the $m$-cyclic cover corresponding to $\mathfrak{L}\in P_m(Y\setminus S)$. Then $\mathfrak{L}$ has order $m$ in $P_m(Y\setminus S)$ if and only if $W_\mathfrak{L}$ is a connected $m$-cyclic cover. For a finite subset $B$ of $P_m(Y\setminus S)$, $B$ is a $\Z/m\Z$-basis of a free submodule of $P_m(Y\setminus S)$ if and only if the set of covers $\{W_{\mathfrak{L}}\longrightarrow Y:\mathfrak{L}\in B\}$ are mutually linearly disjoint, i.e., the fibre product $\times_{\mathfrak{L}\in B} W_\mathfrak{L}\longrightarrow Y$ is an integral scheme. In this case, the normalisation of this fibre product is a $\langle B\rangle$-Galois cover of $Y$ that dominates each connected component of $W_\zeta\longrightarrow Y$ for any $\zeta\in\langle B\rangle$.
\end{lm}
Note that when $B$ is a linearly independent subset of cardinality $n$, $\langle B\rangle\cong(\Z/m\Z)^n$.

\subsection{Equation of Artin-Schreier-Witt curves of degree $p^2$}\label{ASW_eqn}
Let $V\longrightarrow X$ be a Galois cover of degree $p^2$. Then it is given by a reduced Witt vector $(f_1, f_2)$. Then $k(V)=k(X)[x_1,x_2]/I$. When $p$ is odd, the ideal $I$ is generated by $$x_1^p-x_1-f_1\text{ and}$$ 
$$x_2^p-x_2+x_1^{p(p-1)+1}-\frac{\binom{p}{2}}{p}x_1^{p(p-2)+2}+\frac{\binom{p}{3}}{p}x_1^{p(p-3)+3}-\ldots-(-1)^{p-1}x_1^{p+(p-1)}-f_2.$$
When $p=2$, $I=\langle x_1^2-x_1-f_1, x_2^2-x_2-x_1^2+x_1^3-f_2\rangle$ (see \cite[Chapters~8.10~and~8.11]{Jac} or \cite{Wt}).

\subsection{Genus formulas in various cases}
Let us recall the Riemann-Hurwitz formula and Hilbert’s different formula (\cite[Theorem~11.72]{HK}) for calculating the genus of a cover in various cases:
\subsubsection{Genus of tamely ramified covers}\label{genus-tame}
Let $W\longrightarrow V$ be an $H$-Galois cover, where the order of $H$ is coprime to $p$. Let the genera of $W$ and $V$ be $g_W$ and $g_V$ respectively. Since each ramified point in $W$ is tamely ramified, $2g_W-2=(2g_V-2)|H|+\sum_{w\in W} (e_w-1)$, where $e_w$ is the ramification index of $w$ over its image in $V$.

\subsubsection{Genus of Artin-Schreier covers}\label{genus-AS}
Let $X_2\longrightarrow X_1$ be a $\Z/p\Z$-Galois cover given by $z^p-z=F$ for a rational function $F$ in $k(X_1)$ such that $F\neq f^p-f$ for any $f\in k(X_1)$. Then $k(X_2)=k(X_1)[z]/\langle z^p-z-F\rangle$. Let us assume that all the poles of $F$ are of order coprime to $p$. Let $g_1$ and $g_2$ be the genera of $X_1$ and $X_2$, respectively. For any ramified point $y\in X_2$, let $\mathfrak{n}_y$ denote the order of the pole of $F$ at the image of $y$ in $X_1$. By the assumption, $\gcd(\mathfrak{n}_y,p)=1$. Note that the ramified points in $X_2$ are totally ramified and at each such $y$, the $j$-th lower ramification group is $\Z/p\Z$ for $0\le j\le \mathfrak{n}_y$ and others are $\{0\}$. Then $$g_2=\frac{2+p(2g_1-2)+(p-1)\sum_{y\in X_2}(\mathfrak{n}_y+1)}{2}.$$ 

\subsubsection{Genus of Artin-Schreier-Witt covers using upper jumps}\label{genus-u}
Let $\psi:V\longrightarrow \PP^1$ be a $p^a$-cyclic cover, \'etale over $\Aff^1$, and $X_i\longrightarrow \PP^1$ be the $p^i$-cyclic quotient cover of $\psi$, with genus $g_i$, $0\le i\le a$. Here $X_a=V$, $g_0=0$, and $g_a=g_V$.
Then by \cite[Theorem~1]{Gar} $V$ is given by a reduced Witt vector of length $a$ consisting of polynomials (poles only at $\infty$) with degrees $\mathfrak{d}_1, \ldots, \mathfrak{d}_a$ which are coprime-to-$p$ or zero. The upper jumps (i.e. breaks in the upper numbering ramification filtration) above $\infty$ are given by:
$$u_1:=\mathfrak{d}_1,u_2:=\max\{\mathfrak{d}_1p, \mathfrak{d}_2\}, u_3:=\max\{\mathfrak{d}_1p^2, \mathfrak{d}_2p,\mathfrak{d}_3\}, \ldots, u_a:=\max\{\mathfrak{d}_ip^{a-i}: 1\le i \le a\}.$$

It follows immediately that $(u_1, \ldots ,u_a)$ satisfies the following conditions:
\begin{cnd}\label{uJump}
\begin{enumerate*}
\item $u_i\ge 0$;
\item $u_1\ge 1$ and $\gcd(u_1,p)=1$;
\item $u_i\ge pu_{i-1}$, and $u_i=pu_{i-1}$ or $\gcd(u_i,p)=1$ for $i>1$.
\end{enumerate*}
\end{cnd}
Conversely, it follows from Artin-Schreier-Witt theory (\cite[Theorem~8.31]{Jac}) that given any $(u_1, \ldots ,u_a)$ satisfying the above conditions, there are $p^a$-cyclic covers of $\PP^1$ \'etale over $\Aff^1$ whose $i$-th upper jump at $\infty$ is $u_i$, $1\le i\le a$.

The genus of $X_i$ (\cite[Lemma~1]{PR-wild}) is given by 
\begin{align}
g_i &= 1+p^i(0-1)+\frac{p^i\left[p^i-1+(p-1)\sum_{t=1}^ip^{t-1}u_t\right]}{2p^i}\\
&=\frac{1-p^i+(p-1)\sum_{t=1}^ip^{t-1}u_t}{2}.\label{G}
\end{align}

Clearly, when $G\cong \Z/p\Z$, the genus of $X_1$ is $g_1=(p-1)(\mathfrak{d}_1-1)/2$.

\section{$G$-stable decomposition of $P_l(V\setminus S_V)$}

 \begin{notation}\label{basic_not}
 We fix the following notation.
 \begin{itemize}
\item Let $l$ be a prime number other than $p$.
\item Let $a\ge 1$ be an integer and $G$ be a finite cyclic $p$-group of order $p^a$ generated by $\sigma$. 
\item For $1\le s\le a$, $d_s$ denotes the order of $l$ in $(\mathbb{Z}/p^s\mathbb{Z})^*$, the set of units in $\Z/p^s\Z$.
 \item Let $\Phi_{p^s}(x)$ denote the $p^s$-th cyclotomic polynomial, i.e., the minimal polynomial of a primitive $p^s$-th root of unity over $\mathbb{Q}$. Note that $\Phi_{p^s}(x)\in\Z[x]$ and its image in $\FF_l[x]$ factors into $p^{s-1}(p-1)/d_s$ irreducible factors (\cite{LN}, Theorem 2.47).
We label the irreducible factors of $\Phi_{p^s}(x)$ in  $\FF_l[x]$ by $P_{s,t}(x)$, for $1\le t\le p^{s-1}(p-1)/d_s$. Each of these factors has degree $d_s$.
\item Note that $$\Phi_{p^a}(x)=\frac{(x^{p^{a-1}})^p-1}{x^{p^{a-1}}-1}=\Phi_p(x^{p^{a-1}}).$$
 \item The  irreducible $G$-modules over $\FF_l$ are denoted by
 $$M_{s,t}:=\FF_l[x]/P_{s,t}(x),\text{ for }1\le t\le p^{s-1}(p-1)/d_s,~\ 1\le s\le a.$$ Here $\sigma$ acts by multiplication by $x$.
 \item Over $\FF_l$, we have $(x^{p^s}-1)/(x-1)=\prod_{\eta=1}^s \prod_{t=1}^{\phi(p^\eta)/d_\eta}P_{\eta, t}(x)$. Here, the factors are all distinct and irreducible.
 \item Let $H=(\Z/l\Z)^n$ for some $n\ge 1$.
\end{itemize}
\end{notation}

 Note that $\dim(M_{s,t})=\deg(P_{s,t}(x))=d_s$ (by \cite[Lemma 5.1]{EPCh}). Any irreducible non-trivial $G$-representation over $\FF_l$ is isomorphic to $\FF_l[x]/(P_{s,t}(x))$ for some $s$ and $t$. We will use these notations and hypotheses for the rest of this article.
 
Let $\psi:V\longrightarrow X$ be a $G$-Galois cover of smooth connected projective $k$-curves \'etale away from $S_X$, a nonempty finite subset of closed points of $X$ and $S_V=\psi^{-1}(S_X)$. Let $r_X=|S_X|$, $r_V=|S_V|$ and $g_X$, $g_V$ be the genera of $X$, $V$ respectively. Note that $P_l(V\setminus S_V)$ is a (right) $G$-module via the following action: 
 $$ ([L],D)\cdot\sigma = ([\sigma^*L],\sigma^* D), ~\ \sigma\in G.$$ 
It was shown in \cite[Theorem~3.5]{EPCh} that there is a bijective correspondence between $G$-submodules of $P_l(V\setminus S_V)$ isomorphic to $H$ and $H$-covers of $W \longrightarrow V$, \'etale away from $S_V$ such that $W\longrightarrow X$ is a $\Gamma$-Galois cover for some extension $\Gamma$ of $G$ by $H$.
 
\begin{rmk}\label{cover_rmk} \textbf{\textit{Construction of $\Gamma$-covers}:} Since $\Gamma$ is an extension of $G$ by $H$, to construct a connected $\Gamma$-Galois cover of $X$ \'etale away from $S_X$ and dominating $V$, first we need to find an $H$-cover $W$ of $V$, \'etale over $V\setminus S_V$. By Lemma~\ref{C}, it is enough to find $B\subset P_l(V\setminus S_V)$ such that $\chash B=n$ and $\langle B\rangle\cong H$. The composition  $W\longrightarrow V\longrightarrow X$ is a $\Gamma$-Galois cover for some extension $\Gamma$ of $G$ by $H$ if and only if $\langle B\rangle$ is a $G$-stable submodule of $P_l(V\setminus S_V)$ (\cite[Theorem~3.5]{EPCh}).  
\end{rmk} 
For the $G$-Galois cover $\psi: V\longrightarrow X$, the following holds.

\begin{lm}\label{subcover_2}
Let $K$ be a subgroup of $G$ and $Y$ be the normalisation of $V/K$. Then $\P0(V)[l]^{K}=\P0(Y)[l]$.
\end{lm}

\begin{proof}
Since $K\trianglelefteq G$, the cover $\psi$ factors through $h:V\stackrel{K}{\longrightarrow}Y$ and $f:Y\stackrel{G/K}{\longrightarrow} X$. By \cite[Proposition~3.9]{EPCh}, we identify $\P0(Y)[l]$ with $h^*(\P0(Y)[l])\subset\P0(V)[l]^K$. 

Let $\mathcal{L}\in \P0(V)[l]^K$. Then $H=\langle \mathcal{L}\rangle$ is $K$-stable (in fact, $\mathcal{L}\cdot g=\mathcal{L}$ $\forall g\in K$). If $W_\mathcal{L}\longrightarrow V$ is the corresponding $l$-cyclic cover with Galois group $H$ then the composition $W_\mathcal{L}\longrightarrow V\longrightarrow X$ is also Galois (by \cite[Proposition 2.1]{EPCh}). Since $\gcd(l,|K|)=1$, by Schur-Zassenhaus Theorem, we have $\Aut(W_\mathcal{L}|X)=H\rtimes K$. But the $K$-action on $H$ is trivial. Hence $\Aut(W_\mathcal{L}|X)=H\times K$. Then $W_\mathcal{L}/K\stackrel{H}{\longrightarrow}Y$ is a Galois cover. In other words, $\mathcal{L}\in h^*(\P0(Y)[l])=\P0(Y)[l]$.
\end{proof}

\begin{rmk}
From the proof, it is clear that Lemma~\ref{subcover_2} holds more generally. The group $G$ can be an arbitrary finite group instead of $\Z/p^a\Z$, $K$ a normal subgroup of $G$ and $l$ an arbitrary positive integer coprime to $p|K|$.
\end{rmk}

In \cite[Theorem~5.2]{EPCh} we find the values of $n$ for which $\Gamma$-covers of $X$ (where the $G$-action $\theta$ varies), \'etale over $X\setminus S_X$, can be constructed. Let $n_0$ (resp. $n_s$, $1\le s\le a$) be the dimension of $P_l(V\setminus S_V)^G$ (resp. $ \ker(\Phi_{p^s}(\sigma))\subset P_l(V\setminus S_V)$) over $\FF_l$.

\begin{pro}[{{\cite[Theorem~5.2]{EPCh}}}]\label{n_char}
The integer $n$ can be expressed as $n=\gamma_0+\Sigma_{s=1}^a\gamma_sd_s$ for non-negative integers $\gamma_0\le n_0$, $\gamma_s\le n_s/d_s,\forall s\le a$ if and only if there is an $H\rtimes_\theta G$-cover of $X$, \'etale over $X\setminus S_X$, dominating $V$ for some group homomorphism $\theta:G\longrightarrow \Aut(H)$.
\end{pro}

Now we further examine the irreducible decomposition of $P_l(V\setminus S_V)$ to find the values of $n_0$ and $n_t$.

\begin{lm}\label{1lem}
For each $x\in S_X$, $\psi^{-1}(\{x\})$ is a $G$-subset of $S_V$. Moreover, the cardinality of $\psi^{-1}(\{x\})$ divides $p^a$ and the elements in $\psi^{-1}(\{x\})$ can be rearranged as $v, \sigma^{-1}v, \sigma^{-2}v, \ldots$.
\end{lm}
\begin{proof}
It follows from the facts that $\psi:V\longrightarrow X$ is a $G$-Galois cover and that $G$ acts transitively on each fibre.
\end{proof}

\begin{notation}\label{Pram}
Let $P_{\ram}(S_V,l)$ denote the $\Z/l\Z$-module 
\begin{equation}\label{eqn_Pram}
P_{\ram}(S_V,l):=\{\Sigma_{v\in S_V} a_vv \mid a_v\in\mathbb{Z}/l\mathbb{Z}, \Sigma_{v\in S_V} a_v =0\}.
\end{equation}
 Since $S_V$ is a $G$-set, $P_{\ram}(S_V,l)$ is in fact a $G$-submodule of $\mathbb{Z}/l\mathbb{Z}[S_V]$, the $\Z/l\Z$-linear combinations of elements in $S_V$. Similarly, $G$ acts trivially on $\Z/l\Z[S_X]$. The $G$-action is also trivial on the submodule $$P_{\ram}(S_X,l):=\{\Sigma_{x\in S_X} a_xx\mid a_x\in\Z/l\Z, \Sigma_{x\in S_X} a_x =0\}.$$
Note that $P_{\ram}(S_X,l)$ can be identified with the $G$-submodule of $P_{\ram}(S_V, l)$ consisting of 
$$\left\{\Sigma_{x\in S_X}a_x\left(\Sigma_{v\in\psi^{-1}(x)}v\right)\mid a_x\in\Z/l\Z, ~\ \Sigma_{x\in S_X} a_x =0\right\}=(P_{\ram}(S_V, l))^G.$$
\end{notation}

\begin{rmk}
The $G$ action on $P_{\ram}(S_V,l)$ is given by $$(\Sigma_{v\in S_V} a_vv)\cdot\sigma=\Sigma_{v\in S_V} a_v(\sigma^{-1}v).$$
\end{rmk}
\begin{lm}\label{P_ram}
There is an isomorphism $P_l(V\setminus S_V)\cong \P0 (V)[l]\oplus P_{\ram}(S_V,l)$ of $G$-modules.
\end{lm}

\begin{proof}
Consider the exact sequence (\cite[Equation 3.4]{T}) of abelian groups:
\[
\xymatrix{
0\ar[r]&\P0 (V)[l]\ar@{^{(}->}[r]^\iota &P_l(V\setminus S_V)\ar[r]^f &(\mathbb{Z}/l\mathbb{Z})[S_V]\ar[r]^{\overline{\deg}} &\mathbb{Z}/l\mathbb{Z},
}
\]
where $\iota$ is $[L]\mapsto ([L],0)$, $f(([L],D))= D\mod l$ and $\overline{\deg}(D\mod l)=\deg(D) \mod l$. Clearly, $[\sigma^*L]\mapsto ([\sigma^*L],0)=([L],0)\cdot\sigma$. Also, $(f(([L],D)))\cdot\sigma= \sigma^*D\mod l =f(([\sigma^*L],\sigma^*D)) =f(([L],D)\cdot\sigma)$. Since $\sigma$ is an automorphism of $V$, $\deg (D)=\deg(\sigma^*D)$ for any Cartier divisor $D$. Hence $\iota$, $f$ and $\overline{\deg}$ are all $\Z/l\Z$-module homomorphisms and they commute with $\sigma$. So $\iota$, $f$ and $\overline{\deg}$ are $G$-module homomorphisms. Note that $\ker(\overline{\deg})=P_{\ram}(S_V,l)$ by Equation~\eqref{eqn_Pram}. From the exact equation of $G$-modules above and \cite[Theorem~3.4]{S} (since $l \nmid p^a$, $\FF_l[G]$-modules are semisimple), we have a $G$-module isomorphism 
\begin{align*}
P_l(V\setminus S_V)&\cong \P0 (V)[l]\oplus P_l(V\setminus S_V)/\P0 (V)[l] \cong  \P0 (V)[l]\oplus f(P_l(V\setminus S_V)) \\
&\cong \P0 (V)[l]\oplus\ker(\overline{\deg})\\
&\cong \P0 (V)[l]\oplus P_{\ram}(S_V,l).    
\end{align*}
\end{proof}

We fix a $G$-module monomorphism $\imath:P_{\ram}(S_V,l)\hookrightarrow P_l(V\setminus S_V)$ such that $f\circ\imath $ is equal to $ \id|_{P_{\ram}(V\setminus S_V)}$.

\begin{rmk}
Clearly, the $l$-cyclic covers of $V$ corresponding to elements in $\P0 (V)[l]$ are \'etale over $V$.
If $\imath(\Sigma_{v\in S_V} a_vv)=([L],\Sigma_{v\in S_V} a_vv)$, the corresponding $l$-cyclic cover of $V$ is ramified at $v\in S_V$ if and only if $a_v\not\equiv 0\mod l$. In this case, the ramification index at $v$ is $l$ (see the discussion after the definition of $P_m(Y\setminus S)$). 
\end{rmk}

\begin{notation}\label{C_i^X}
We stratify $S_X$ based on $\psi:V\longrightarrow X$.
\begin{enumerate}
 \item For $0\le i \le a$, let $C^X_i$ be the possibly empty subset of points $x\in S_X$ such that there are exactly $p^i$ points in $\psi^{-1}(x)$. Let $r_i=\chash C^X_i$. Note that $r_0$ is the number of totally ramified points in $X$ and $r_a$ is the number of unramified points in $S_X$. Then $S_X$ is the disjoint union of the subsets $C^X_i$. Clearly,  $r_X=\sum r_i$ and $r_V=\sum_{i=0}^a r_ip^i$. 
 
For $x_{i,j}\in C^X_i$, $1\le j\le r_i$, let $\psi^{-1}(\{x_{i,j}\})=\{v_{i,j,1},\ldots,v_{i,j,p^i}\}$ be labelled so that they are permuted in order by $\sigma$ (Lemma~\ref{1lem}). 
\item \label{Vij} Fix an index $I\in\{0, 1,\ldots, a\}$ such that $r_I\neq 0$ is non-empty. Let $r_X>1$. For $(i,j)$ with $0\le i\le a$, $r_i>0$, $1\le j\le r_i$ and $(i,j)\neq (I,1)$, let $V'_{i,j}\longrightarrow X$ be the (connected) $l$-cyclic cover corresponding to the element $x_{i,j}-x_{I,1}$ of $P_{\ram}(S_X,l)\into P_l(X\setminus S_X)$ (by Lemma~\ref{C}). 

Let $V_{i,j}\longrightarrow V$ be the normalization of $V'_{i,j}\times_X V$. These notations hold till the rest of the paper.
\item \label{Wij-Bij} Let $r_i\neq 0$ for some $i$. Let $B_{ij}\subset P_{\ram}(S_V,l)$ be defined by $B_{ij}=\{v_{i,j,1}-v_{i,j,2}, v_{i,j,2}-v_{i,j,3},\ldots, v_{i,j,p^{i}-1}-v_{i,j,p^i} \}$. Note that $B_{i,j}$ is a set of linearly independent elements and also a $G$-set.

Let $W_{i,j}$ denote the subspace of $P_{\ram}(S_V,l)$ generated by the basis $B_{ij}$. Each $W_{i,j}$ is a $G$-stable subspace of dimension $p^i-1$.
\end{enumerate}
\end{notation}

\begin{lm}\label{D}
Let the notations be as in Notation~\ref{C_i^X}(\ref{Vij}). For $(i,j)$ with $0\le i\le a$, $r_i>0$, $1\le j\le r_i$, let an element $D_{i,j}$ in $(\Z/l\Z)[S_V]$ be defined by
\begin{equation}\label{eqn_D}
D_{i,j}=
\begin{cases}
\sum_{u=1}^{p^i} v_{i,j,u} - p^{i-I}\sum_{u=1}^{p^I} v_{I,1,u},~\  I\le i,\\
p^{I-i}\sum_{u=1}^{p^i} v_{i,j,u} - \sum_{u=1}^{p^I} v_{I,1,u},~\ I>i.
\end{cases}
\end{equation}

Then $D_{i,j}\in P_{\ram}(S_V,l)$, $\langle \imath(D_{i,j})\rangle$ is a 1-dimensional $G$-submodule of $P_l(V\setminus S_V)$ with trivial $G$-action and it corresponds to the $l$-cyclic cover $V_{i,j}\longrightarrow V$.
\end{lm}

\begin{proof}
Since the sum of the coefficients of $D_{i,j}$ is $0$, $D_{i,j}$ is an element of $P_{\ram}(S_V,l)$. Since $\{v_{i,j,1},\ldots,v_{i,j,p^i}\}$ is the preimage of $x_{i,j}$ under $\psi$, it is a $G$-set and $G$ acts trivially on $D$. So $\langle \imath(D_{i,j})\rangle$ is a 1-dimensional $G$-submodule of $P_l(V\setminus S_V)$ with trivial $G$-action and (by Remark~\ref{cover_rmk} or \cite[Theorem~3.5]{EPCh}) we get a (connected) $(\Z/l\Z)\times G$-cover of $X$ dominating $V$. 
Since the image of $x_{i,j}-x_{I,1} \in P_{\ram}(S_X,l)$ under $\psi^*$ is $D_{i,j}$, this Galois cover is $V_{i,j}\longrightarrow X$.
\end{proof}

\begin{lm}\label{ram}The following holds:
\begin{enumerate}
 \item The $\Z/l\Z$-cover $V_{i,j}\longrightarrow V$ is ramified at each $v_{I,1,u}$, $1\le u\le p^I$, and each $v_{i,j,u}$, $1\le u\le p^i$, and \'etale elsewhere. 

 \item For $i>0$ and any irreducible $G$-submodule $M$ of $W_{i,j}$, the corresponding $M$-cover of $V$ will be ramified at each point in $\psi^{-1}(x_{i,j})=\{v_{i,j,1},\ldots, v_{i,j,p^i}\}$ and \'etale elsewhere.
\end{enumerate}
\end{lm}

\begin{proof}
The cover $V_{i,j}$ corresponds to $([L],D_{i,j})\in\imath(P_{\ram}(S_V,l))$, where $f(([L],D_{i,j}))=D_{i,j}\mod l\in P_{\ram}(S_V,l)$. Since the coefficient of each $v_{i,j,u}$ (resp. $v_{I,1,u}$) is 1 or $p^{I-i}$ (resp. $p^{I-1}$ or 1), it is nonzero in $\Z/l\Z$. The part (1) follows from \cite[3.15. Lemma]{EV}) (see the discussion after the definition of $P_m(Y\setminus S)$).

Let $w$ be a nonzero element of $M\subset W_{i,j}=\langle B_{i,j}\rangle$. Then $w=\sum_{u=1}^{p^i}a_{i,j,u}v_{i,j,u}$ for the coefficients $a_{i,j,u}\in\Z/l\Z$. Then at least one coefficient, say $a_{i,j,\kappa}$, is nonzero for some $\kappa$ between 1 and $p^i$. 
Then $a_{i,j,\kappa}v_{i,j,\kappa+t}$ is present in the expression of $w\cdot\sigma^t$, where the index $\kappa+t$ is taken modulo $p^i$. Note that $\{\kappa+t \text{ modulo }p^i :1\le t\le p^i\}=\{1,\ldots,p^i\}$. Then the cover corresponding to $\imath(w\cdot\sigma^t)$ is ramified at $v_{i,j,{\kappa+t}}$. Since $G$ acts transitively on $\psi^{-1}(x_{i,j})$, $M$ is $G$-stable and the $M$-cover of $V$ dominates the cover corresponding to every element in $M$ (by Lemma~\ref{C}), the $M$-cover is ramified at each point in $\psi^{-1}(x_{i,j})$.
\end{proof}
 Let $M_{s,t}$ be as in Notation~\ref{basic_not}.
\begin{thm}\label{2lem}
Let $S_X$ be non-empty. For $0\le i\le a$, let $r_i$ be the number of points $x$ in $S_X$ such that there are exactly $p^i$ points in $\psi^{-1}(x)$. There is a $G$-module isomorphism $$P_{\ram}(S_V,l) \cong \mathbb{F}_l^{r_X-1}\oplus(\oplus_{s=1}^a \oplus_{t=1}^{p^{s-1}(p-1)/d_s}M_{s,t}^{\sum_{i=s}^a r_i}).$$
In particular,
 $$P_l(V\setminus S_V) \cong \P0 (V)[l]\oplus\mathbb{F}_l^{r_X-1}\oplus(\oplus_{s=1}^a \oplus_{t=1}^{p^{s-1}(p-1)/d_s}M_{s,t}^{\sum_{i=s}^a r_i})$$
\end{thm} 

\begin{proof}
In view of Lemma~\ref{P_ram}, it is enough to show that
$$P_{\ram}(S_V,l)\cong\mathbb{F}_l^{r_X-1} \oplus (\oplus_{s=1}^a \oplus_{t=1}^{p^{s-1}(p-1)/d_s} M_{s,t}^{\sum_{i=s}^a r_i})$$ as $G$-modules. As a vector space, $P_{\ram}(S_V,l)$ has dimension $\chash(S_V)-1=r_V-1$ over $\FF_l$. From Equation~\ref{D}, for $(s,\nu)$ with $0\le s\le a$, $r_s>0$ and $1\le\nu\le r_s$, we have 
\[D_{s,\nu}=\begin{cases}
\sum_{u=1}^{p^s} v_{s,\nu, u} -\left(p^{s-I}\right)\sum_{u=1}^{p^I} v_{I,1,u},~\  I\le s,\\
\left(p^{I-s}\right)\sum_{u=1}^{p^s} v_{s,\nu, u} - \sum_{u=1}^{p^I} v_{I,1,u},~\ I>s.
\end{cases}\]
Now $B_0=\{D_{s,\nu}\mid 0\le s\le a, 1\le \nu\le r_s,(s,\nu)\neq (I,1)\}$ is a set of linearly independent elements and each element generates a $G$-submodule of $P_{\ram}(S_V,l)$ with trivial action, by Lemma~\ref{D}. Hence, the number of $G$-components of $P_{\ram}(S_V,l)$, with trivial action, is at least $\chash B_0=r_X-1$.

The dimension of $W_{s,\nu}$ is $p^s-1$ and $W_{s,\nu}\cap (\oplus_{(s,\nu)\neq(s',\nu')}W_{s',\nu'})=\{0\}$. Since $B_0$ and $B_{s\nu}$'s are disjoint and there union is a set of linearly independent elements, $\dim(\oplus_{s=1}^a \oplus_{\nu=1}^{r_s} W_{s,\nu})$ $=\sum_{s=1}^a r_s(p^s-1)$. Also, $\dim(P_{\ram}(S_V,l))= r_V-1=(\sum_{s=1}^a r_sp^s) - 1=r_X-1+\sum_{s=1}^a r_s(p^s-1)$. Then we have the following $G$-module decomposition:
\begin{equation}\label{eqn_3.14}
P_{\ram}(S_V,l)\cong \mathbb{F}_l^{r_X-1}\oplus(\oplus_{s=1}^a \oplus_{\nu=1}^{r_s} W_{s,\nu}).
\end{equation}
 
Since $(v_{s,\nu, u}-v_{s,\nu,u+1})\cdot\sigma=v_{s,\nu,u+1}-v_{s,\nu,u+2}$ 
(the index $u$ is taken modulo $p^s$), the characteristic polynomial of $\sigma_{|W_{s,\nu}}$ is $(x^{p^s}-1)/(x-1)=\prod_{\eta=1}^s \prod_{t=1}^{\phi(p^\eta)/d_\eta}P_{\eta, t}(x)$. From its factorization (see Notation~\ref{basic_not}), we have the $G$-module decomposition: 

\begin{equation}\label{W_ij-decomposition}
W_{s,\nu}\cong \oplus_{\eta=1}^s \oplus_{t=1}^{\phi(p^\eta)/d_\eta}
(\mathbb{F}_l[x]/(P_{\eta, t}(x))= \oplus_{\eta=1}^s \oplus_{t=1}^{\phi(p^\eta)/d_\eta}M_{\eta,t}.
\end{equation}
For $1 \le \nu \le r_\eta$, the multiplicity of $M_{\eta t}$ in $W_{s,\nu}$ is 1 if $s \ge \eta$, the multiplicity is zero if $s < \eta$. From Equations~\ref{eqn_3.14} and \ref{W_ij-decomposition},

\begin{align*}
P_{\ram}(S_V,l)\cong\FF_l^{r_X-1}\oplus&(\oplus_{s=1}^a \oplus_{\nu=1}^{r_s} \oplus_{\eta=1}^s \oplus_{t=1}^{\phi(p^\eta)/d_\eta} M_{\eta,t})\\
\cong\FF_l^{r_X-1}\oplus&\left[\oplus_{s=1}^a \oplus_{\nu =1}^{r_s}\left\{\left(\oplus_{t=1}^{\phi(p)/d_1}M_{1,t}\right)\oplus
\ldots\oplus\left(\oplus_{t=1}^{\phi(p^s)/d_s}M_{s,t}\right)\right\}\right]\\ 
\cong \FF_l^{r_X-1}\oplus&\left(\oplus_{t=1}^{\phi(p)/d_1}M_{1,t}^{r_1}\right)\\
&\oplus\left[\left(\oplus_{t=1}^{\phi(p)/d_1}M_{1,t}^{r_2}\right)\oplus\left(\oplus_{t=1}^{\phi(p^2)/d_2}M_{2,t}^{r_2}\right) \right]\\
&\oplus\ldots\\
&\oplus\left[\left(\oplus_{t=1}^{\phi(p)/d_1}M_{1,t}^{r_a}\right)\oplus\left(\oplus_{t=1}^{\phi(p^2)/d_2}M_{2,t}^{r_a}\right)\oplus\ldots\oplus\left(\oplus_{t=1}^{\phi(p^a)/d_a}M_{a,t}^{r_a}\right) \right]\\
\cong \FF_l^{r_X-1}\oplus&( \oplus_{s=1}^a \oplus_{t=1}^{\phi(p^s)/d_s}M_{s,t}^{\sum_{\nu=s}^a r_\nu}).
\end{align*}
We put $\phi(p^s)=p^{s-1}(p-1)$ and replace $\nu$ by $i$ to get the required equation.
\end{proof}

\begin{notation}
Let $X_s$ denote the normalization of $V/\langle \sigma^{p^s}\rangle$ for $0\le s\le a$ and $g_s$ denote the genus of $V_s$. Let $g_{-1}:=0$. Here $X_0=X$ and $X_a=V$.
\end{notation}
Let the restriction of $\sigma:P_l(V\setminus S_V)\longrightarrow P_l(V\setminus S_V)$ on ${\P0(V)[l]}$ be also denoted by $\sigma$.
\begin{lm}\label{genus_subcover}
The dimension of $\ker\left(\Phi_{p^s}(\sigma):{\P0(V)[l]}\longrightarrow {\P0(V)[l]}\right)$ over $\FF_l$ is $2g_s-2g_{s-1}$ for $0\le s\le a$.
\end{lm}

\begin{proof}
Let $\tau=\sigma^{p^{a-1}}$ which is an automorphism of order $p$. Consider the action of $\tau$ on $V$ and $\P0(V)[l]$. Consider the intermediate $p$-cyclic cover $V\longrightarrow X_{a-1}$ where $X_{a-1}$ is the normalization of $V/\langle\tau\rangle$. Since $\ker(\tau-\id)=\P0(X_{a-1})[l]$ (Lemma~\ref{subcover_2}) and the minimal polynomial of $\tau$ divides $x^p-1=(x-1)\Phi_p(x)$, we have $\P0(V)[l]=\P0(X_{a-1})[l]\oplus \ker(\Phi_p(\tau))$. Since
$\Phi_{p^a}(x)=\Phi_p(x^{p^{a-1}})$ (see Notation~\ref{basic_not}), $\ker(\Phi_p(\tau))=\ker(\Phi_{p^{a}}(\sigma))\subset \P0(V)[l]$ and it has dimension $2g_a-2g_{a-1}=2g_V-2g_{a-1}$.

Let $\overline{\sigma}$ be the image of $\sigma$ in the quotient group $\Aut(X_{a-1}|X)$, a $p^{a-1}$-cyclic group generated by $\overline{\sigma}$. We also denote the restriction of $\overline{\sigma}$ on $\P0(X_{a-1})[l]$ by $\overline{\sigma}$. As before, we see that the dimension of the kernel of $\Phi_{p^{a-1}}(\overline{\sigma}):\P0(X_{a-1})[l]\longrightarrow\P0(X_{a-1})[l]$ is $2g_{a-1}-2g_{a-2}$. But as linear operators on $\P0(X_{a-1})[l]$, $\overline{\sigma}=\sigma$. Hence $\ker(\Phi_{p^{a-1}}({\sigma}))$ has dimension $2g_{a-1}-2g_{a-2}$. Now the result follows from induction.
\end{proof}

The result below follows directly from Theorem~\ref{2lem} and Lemma~\ref{genus_subcover}.

\begin{co}\label{dim-ker}
Let $S_X$ be non-empty. The dimension $n_s$ of $\ker(\Phi_{p^s}(\sigma))\subset P_l(V\setminus S_V)$ is $$n_s=2g_s-2g_{s-1}+p^{s-1}(p-1)\sum_{i=s}^a r_i \text{ for }1\le s\le a.$$
The dimension of $ P_l(V\setminus S_V)^G$ is $n_0=2g_X+r_X-1$.
\end{co}

\begin{proof}
 By Lemma~\ref{P_ram}, $P_l(V\setminus S_V)=P_{\ram}(S_V,l)\oplus \P0(V)[l]$.
 Note that the kernel of $\Phi_{p^s}(\sigma)$ restricted to $P_{\ram}(S_V,l)$  is $\oplus_{t=1}^{p^{s-1}(p-1)/d_s}(\FF_l[x]/P_{s,t}(x))^{\sum_{i=s}^a r_i})$ (Theorem~\ref{2lem}) and its dimension is $p^{s-1}(p-1)\sum_{i=s}^a r_i$. Now use Lemma~\ref{genus_subcover}.
\end{proof}

\begin{rmk}
The values of $n_s$, $0\le s\le a$, are independent of the prime number $l$.
\end{rmk}

\begin{ex}\label{ex1}
Let $X$ be $\PP^1_k=\Spec k[x]\cup \Spec k[\frac{1}{x}]$, $G=\langle\sigma\rangle\cong\Z/9\Z$. Hence $p=3$ and $a=2$. Let $l=2$ and $S_X=\infty$. Let $\psi:V\longrightarrow\PP^1_k$ 
be the $G$-cover given by a Witt vector $(f_1,f_2)$ for polynomials $f_1$ of degree 2, $f_2$ of degree $4$ in k[x]. We will show that there exists a $(\Z/2\Z)^n\rtimes (\Z/9\Z)$-cover of $\PP^1_k$, \'etale over $\Aff^1_k$, dominating $V$, for some group action $G\to\Aut((\Z/2\Z)^n)$ if and only if $0<n=6, 12, 18,24,30; 2,8,14,20, 26, 32$.

As discussed in Subsection~\ref{ASW_eqn}, $k(V)=k(x)[x_1,x_2]/I$ where the ideal $I$ is generated by $x_1^3-x_1=f_1$,
$x_2^3-x_2+x_1^{7}-x_1^{5}=f_2$.
 Also $x_1^{7}-x_1^{5}-f_2\neq g_2^3-g_2$ for any $g_2\in k(x,x_1)$ because the Witt vector $(f_1,f_2)$ induces the $\Z/9\Z$-Galois  extension $k(V)/k(x)$. Let $X_1\longrightarrow \PP^1_k$ be the intermediate cover, with function field $k(x,x_1)$ ($\subset k(V)$).
We use the Riemann-Hurwitz formula and the Hilbert different formula to calculate the genera of $X_1$ and $V$.
Clearly, $V\longrightarrow \PP^1_k$ is ramified only at the preimage of $\infty$ in $X_1$. Since $f_1$ has pole of order $2$ only at $\infty$, the genus of $X_1$ is $g_1=(3-1)(2-1)/{2}=1$ (see Subsection~\ref{genus-u}). Note that $x$ and $x_1$ has poles of respective orders $3$, $2$ at $\infty$. Clearly, $x_1^{7}-x_1^{5}-f_2$ has pole of order $14$ at the point in $X_1$ above $\infty$ and so the genus of $V$ is $g_V=[2+3(2g_1-2)+(14+1)(3-1)]/2=16$ (see Subsection~\ref{genus-AS}). Note that here $d_1=2$, $d_2=6$. 

By Lemma~\ref{genus_subcover}, $n_0=0$, $n_1=2g_1-0=2$ and $n_2=2g_V-2g_1=30$. By Proposition~\ref{n_char}, there exists a $(\Z/2\Z)^n\rtimes_\theta (\Z/9\Z)$-cover of $\PP^1$ that dominates $V$ for some $G$-action $\theta$ if and only if $n=0+2\gamma_1+6\gamma_2$, $\gamma_1\le 2/2=1$, $\gamma_2\le 30/6=5$. Hence there exists a $(\Z/2\Z)^n\rtimes (\Z/9\Z)$-cover of $\PP^1_k$, \'etale over $\Aff^1_k$, dominating $V$, for some group action $G\to\Aut((\Z/2\Z)^n)$ if and only if the values of $n$ are as mentioned above.
\end{ex}

\section{Minimal genus of Galois covers dominating a cyclic $p$-group cover}
The hypothesis and notations are as in the previous section (see Notation~\ref{basic_not} and \ref{C_i^X}). Recall that $G\cong \Z/p^a\Z$ is generated by $\sigma$, $H\cong (\Z/l\Z)^n$ and $\theta:G\longrightarrow \Aut(H)$ is a $G$-action. Throughout this section, $a$, $n$ and $\theta$ will remain fixed. Let $\psi: V\longrightarrow X$ be the $G$-cover as in the previous section. 

Let $\Psi:W\longrightarrow V$ be an $H$-Galois cover such that the composition $W\longrightarrow X$ is an $H\rtimes_\theta G$-cover. Let $\mathcal{B}_{\Psi}$ be the branch locus of $\Psi$ in $V$. Let $B^W_i:= C^X_i \cap \psi(\text{ Branch locus of } \Psi)$, for $0\le i\le a$. These are possibly empty subsets of $X$. Note that the image of $\mathcal{B}_{\Psi}$ in $X$ is the disjoint union of $B_i^W$. 
\begin{pro}\label{genus}
The genus of $W$ is
$$g_W=\frac{2+|H|(2g_V-2)+l^{n-1}(l-1)(\sum_{i=0}^a|B_i^W|p^i)}{2}.$$
  
\end{pro}
\begin{proof}
Note that for any $x\in B^W_i$ and for every $w\in (\psi\circ\Psi)^{-1}(\{x\})$, the ramification index $e(w|\Psi(w))=l$. Hence $\chash(\Psi^{-1}(\Psi(w)))=|H|/l=l^{n-1}$. Since there are $p^i$ points in $S_V$ over each $x\in B_i^W$ and $l^{n-1}$ points in $W$ over each $v\in \psi^{-1}(\{x\})$, by Riemann-Hurwitz formula (see Subsection~\ref{genus-tame}), 
$2g_W-2=|H|(2g_V-2)+l^{n-1}(\sum_{i=0}^a|B_i^W|p^i)(l-1).$
\end{proof}

Now we can calculate the minimal genus of $H\rtimes_\theta G$-covers of $X$ dominating $\psi:V\longrightarrow X$ in specific cases.

 For $t\ge 0$, recall that $M_{1,t}$'s are the (non-isomorphic) non-trivial irreducible representations of $G$ over $\FF_l$ (see Notation~\ref{basic_not}). Let $M_0$ be the $G$-representation with trivial $G$-action. Note that $H$ is a $G$-module (via $\theta:G\longrightarrow \Aut(H)$). 

\begin{thm}\label{min_gen}
Let $G\cong \Z/p^a\Z$, $H\cong (\Z/l\Z)^n$, $\theta$ be a $G$-action on $H$. Let $a$, $n$ and $\theta$ be fixed. Let there be a $G$-module monomorphism from $H$ into $P_l(V\setminus S_V)$. Then there exist $H\rtimes_{\theta}G$-covers of $X$, dominating $V$. Let $H'$ be a maximal $G$-submodule of $H$ such that $H'$ can be embedded in $\P0(V)[l]$
.
Let $\alpha_0$ (resp. $\alpha_{1,t}$, $t\ge 1$) be the multiplicity of $M_0$ (resp. $M_{1,t}$) in $H/H'$.

The minimum of genera of $H\rtimes_\theta G$-covers of $X$, \'etale over $X\setminus S_X$, dominating $V\longrightarrow X$ is $g_{\min}=[2+|H|(2g_V-2)+l^{n-1}(l-1)\mathfrak{c}]/2$ where $\mathfrak{c}$ is as follows:
\begin{enumerate}
    \item Let $S_X=C^X_0$, i.e. each $v$ in $S_V$ is totally ramified. Then $\mathfrak{c}=\alpha_0+1.$

    \item Let $S_X=C^X_1$, i.e. for each 
    $x \in S_X$, let there be $p$ elements in $\psi^{-1}(x)$. Let $\alpha_0\neq 0$. Then $\mathfrak{c}=p\max\{\alpha_0+1,\alpha_{1,1},\alpha_{1,2}\ldots\}.$
    
\end{enumerate}

\end{thm}

\begin{proof}
Let $\Psi:W\longrightarrow V$ be an $H$-Galois cover such that the composition $\psi\circ\Psi:W\longrightarrow X$ is a  $H\rtimes_\theta G$-Galois cover of $X$, \'etale away from $S_X$. 
Hence by \cite[Theorem~3.5]{EPCh} the cover $\Psi$ corresponds to a $G$-submodule of $P_l(V\setminus S_V)$ isomorphic to $H$. We denote this $G$-submodule of $P_l(V\setminus S_V)$ also by $H$.  
The genus of $W$ is determined by the number of ramification points of the $H$-cover $W\longrightarrow V$, i.e., it depends on $|\mathcal{B}_{\Psi}|$. If $W$ is of minimum genus among all such curves, then $H\cap \P0(V)[l]=H'$. This is because the elements in $\P0(V)[l]$ correspond to \'etale covers of V, and ramified covers have higher genus. Let $H''$ be a $G$-submodule of $H$ such that $H=H'\oplus H''$. Then $H''\cong H/H'$. To ensure $W$ has minimal genus, $H''\subset P_{\ram}(S_V,l)$ should be such that the number of ramification points $|\mathcal{B}_{\Psi}|=\sum_{i=0}^a|B_i^W|p^i$ for the cover $W\longrightarrow V$ is minimum, by Proposition~\ref{genus}. Here $B_i^W$ is as in Proposition~\ref{genus}. Now we work out each scenario separately.

 In (1), 
$S_X=C_0^X=\{x_1,\ldots,x_{r_0}\}$, $P_{\ram}(S_V,l)\cong \FF_l^{r_0-1}$ (Theorem~\ref{2lem}) and hence $H''\cong M_0^{\alpha_0}\cong\FF_l^{\alpha_0}$. Here, the $G$-action is trivial. Note that $\alpha_0+1\le r_0$. Put $H''\cong \langle x_2-x_1,\ldots, x_{\alpha_0+1}-x_1\rangle$ so that for the resultant $H'\oplus H''$-cover $W\longrightarrow V$, $B_0^W=\{x_1,\ldots, x_{\alpha_0+1}\}$. Then the $H$-cover $W\longrightarrow V$ has the required genus. For any other $H$-cover $W'\longrightarrow V$ satisfying the given conditions, $|B_0^{W'}|\ge \dim(H'')+1= \alpha_0+1$. Therefore, the genus of $W$ above is minimal.
 
 In (2),
$S_X=C^X_1=\{x_1,\ldots x_{r_1}\}, r_X=r_1, \chash(S_V)=pr_1$, and by Theorem~\ref{2lem} $$P_{\ram}(S_V,l)\cong \mathbb{F}_l^{r_1-1}\oplus[\oplus_{t=1}^{(p-1)/d_1} M_{1,t}^{r_1}]$$ 
and hence $$H''\cong 
\mathbb{F}_l^{\alpha_0}\oplus[\oplus_{t=1}^{(p-1)/d_1} M_{1,t}^{\alpha_{1,t}}].$$ 
Let $A=\max\{\alpha_0+1,\alpha_{1,1},\alpha_{1,2},\ldots\}$. Choose $\langle x_2-x_1,\ldots, x_{\alpha_0+1}-x_1\rangle \cong\mathbb{F}_l^{\alpha_0}$. Let $W_{i,j}$ be as in Notation~\ref{C_i^X}(\ref{Wij-Bij}).
For $1\leq j\leq \alpha_{1,t}$, choose the irreducible submodule isomorphic to $M_{1,t}$ from a fixed decomposition of $W_{1,j}\cong \oplus_{\tau=1}^{(p-1)/d_1}\FF_l[x]/P_{1,\tau}(x)$ (see \eqref{W_ij-decomposition} in the proof of Theorem~\ref{2lem}). This will give $G$-stable submodules isomorphic to $M_{1,t}^{\alpha_{1,t}}$. From these and $H'\subset \P0(V)[l]$, we form an $H$-cover $W$ of $V$, and $B^W_1=\{x_1,\ldots, x_A\}$.  Then the genus of $W$ is $g_{\min}$.

Let $W'$ be another $H\rtimes_\theta G$-cover of $X$ dominating $V$, \'etale away from $S_X$. Then $B^{W'}_1$ is a subset of $ S_X$. We may assume that the $G$-submodule $H$ of $P_l(V\setminus S_V)$ corresponding to the $H$-cover $W'\rightarrow V$ satisfies $H\cap \P0(V)[l]=H'$. Then the genus of $W'$ is $g_{W'}=[2+|H|(2g_V-2)+l^{n-1}(l-1)(|B_1^{W'}|p)]/2$. Note that the summand $\mathbb{F}_l^{\alpha_0}$ of $H''$ ensures that there are at least $\alpha_0+1$ points in $B_1^{W'}$.
Since the multiplicities of $M_{1,t}$ in $H''$ and $P_{\ram}(\psi^{-1}(B_1^{W'}),l)$ are $\alpha_{1,t}$ and $\chash(B_1^{W'})$ respectively, and $H''\into P_{\ram}(\psi^{-1}(B_1^{W'}),l)$, it follows that $\#(B_1^{W'})\geq \alpha_{1,t}$. Therefore $g_{W'}\geq  g_{\min}$.

\end{proof}

Let the hypothesis and the notations be as in the Proposition~\ref{genus}. Let $M_{s,t}$ denote the irreducible $G$-module $\FF_l[x]/P_{s,t}(x)$ for $1\le s\le a$, $1\le t\le p^{s-1}(p-1)/d_s$ (where $\sigma$ acts by multiplication by $x$) as in Notation~\ref{basic_not}. Let $W_{i,j}$ be as in Notation~\ref{C_i^X}(\ref{Wij-Bij}). For $1\le j\le r_s$, the multiplicity of $M_{s,t}$ in $W_{i,j}$ is $1$ if $i\ge s$, the multiplicity is zero if $i<s$ (see equation \eqref{W_ij-decomposition}). 

Note that $H\cong (\Z/l\Z)^n$ is a $G$-module via the $G$-action $\theta$ and $H'$ is the largest possible $G$-submodule of $H$ that can be embedded in $\P0(V)[l]$. Let the multiplicity of the $G$-representation $\FF_l$ with trivial $G$-action (resp. $M_{s,t}$) in $H/H'$ be $\alpha_0$ (resp. $\alpha_{s,t}$). 

\begin{proof}[Proof of Theorem~\ref{min_alg}]
Let $H''$ be a $G$-submodule of $H$ such that $H=H'\oplus H''$ as $G$-modules. Then $H''\cong H/H'$.
Let $\Psi:W\longrightarrow V$ be an $H$-cover such that $W\longrightarrow X$ is an $H\rtimes_{\theta} G$-cover of $X$. Since we are looking for $W$ with minimum genus, we may assume that the corresponding $G$-submodule $H$ of  $P_{l}(V\setminus S_V)$ (by \cite[Theorem~3.5]{EPCh} ) satisfies $H'=H\cap \P0(V)[l]$ and $H''=H\cap P_{\ram}(S_V,l)$ (since the direct summand of $H$ embedded into $\P0(V)[l]$ does not contribute to the degree of the ramification divisor of $\Psi$). By Proposition~\ref{genus}, the genus of $W$ depends on the cardinality $\sum_{s=0}^a p^s|B^W_s|$ of the (tamely) branched points of $\Psi$ in $V$. Put $y_s=|B^W_s|$. Then we have to minimise $\sum_{s=0}^ap^sy_s$ for integer values of $y_s$. The constraints on $y_s$ in the statement of the theorem come from the constraints on $B^W_s$ as we show below.

Let $R(i,j)=\{(i,j)\mid x_{i,j}\in\sqcup_{i=1}^aB^W_i\}$. By Lemma~\ref{ram}, the non-trivial part $\oplus_s\oplus_{t}M_{s,t}^{\alpha_{s,t}}$ of $H''$ is embedded into $\oplus_{(i,j)\in R(i,j)}W_{i,j}$. Also, for $s>0$, $M_{s,t}^{\alpha_{s,t}}$, a subspace of $\ker (\Phi_{p^s}(\sigma|_{H''}))$, is embedded into $\ker(\Phi_{p^s}(\sigma|_{P_{\ram}(S_V,l)}))=\oplus_{i=s}^a \oplus_{j=1}^{r_s}W_{ij}$. Hence the image of $M_{s,t}^{\alpha_{s,t}}\subset H''$ must be in $$(\oplus_{R(i,j)}W_{i,j})\cap(\oplus_{i=s}^a \oplus_{j=1}^{r_i}W_{i,j})=\oplus_{R(i,j)}W_{i,j}.$$ Note that the multiplicity of $M_{s,t}$ in $\oplus_{R(i,j)}W_{i,j}$ is $\sum_{i=s}^a|B^W_i|$. Hence $\sum_{i=s}^a|B^W_i|\ge\alpha_{s,t}$, the multiplicity of $M_{s,t}$ in $H''$, for each $t\in\{1, \ldots, (p^{s-1}(p-1))/d_s\}$. This provides the constraint $\sum_{i=s}^a y_s\ge \max_t\{\alpha_{s,t}\}$.

Clearly, if $\alpha_0>0$, then $\FF_l^{\alpha_0}\into P_{\ram}(S_V,l)$ factors through $\FF_l^{\alpha_0}\into P_{\ram}(S_X,l)$ and the corresponding $\FF_l^{\alpha_0}$-cover of $X$ (dominated by $W\longrightarrow X$) is ramified over at least $\alpha_0+1$ points in $\sqcup_{s=0}^a B^W_s$. Hence, $\sum_{i=0}^a|B^W_i|\ge\alpha_0+1$. Since $B^W_s\subset C^X_s$, $|B^W_s|\le r_s$. Hence the constraints $\sum_{s=0}^ay_s\ge \alpha_0+1$ and $0\le y_s\le r_s$.

Suppose the minimum value $c_{\min}$ is attained at $y_s=\lambda_s$. For $s\ge 0$, we choose subsets $\Lambda_s$ of cardinality $\lambda_s$ in $C^X_s$. For $s\ge 1$, we can define $M_{s,t}^{\alpha_{s,t}}\into \oplus_{x_{i,j}\in\sqcup_{i=s}^a\Lambda_i}W_{i,j}$ (since the multiplicity of $M_{s,t}$ in the latter is $\sum_s^a\lambda_i\ge \max_t\{\alpha_{s,t}\}$). For $\FF_l^{\alpha_0}\into P_{\ram}(S_X,l)\subset P_{\ram}(S_V,l)$, we choose $\alpha_0+1$ points $\{u_0,\ldots,u_{\alpha_0}\}$ in $\sqcup_{i=0}^a\Lambda_i$ and define $\FF_l^{\alpha_0+1}\stackrel{\cong}{\longrightarrow}$ $\langle u_i-u_0\mid 1\le i\le \alpha_0\rangle$. 
Also choose a $G$-submodule of $\P0(V)[l]$ isomorphic to $H'$. Taking their sum we get a $G$-submodule of $P_l(V\setminus S_V)$ isomorphic to $H$. 

This will give us an $H$-cover $\Psi':W'\longrightarrow V$ (by \cite[Theorem~3.5]{EPCh}) and the image of its branch locus in $X$ will be contained in $\sqcup_{i=0}^a\Lambda_i$ (by Lemma~\ref{ram}). In fact, the image is exactly $\sqcup_{i=0}^a\Lambda_i$ as $\sum_0^ap^s\lambda_s$ is minimum. Clearly, the genus of $W'$ is $\mathfrak{g}$ by the Riemann-Hurwitz formula, and it is the minimum.
\end{proof}

\begin{ex}
Let $X$ be $\PP^1_k$ and $x$ be a local coordinate. Let $p=5$, $l=11$, and $a=2$. Hence $G=\langle\sigma\rangle \cong \Z/25\Z$. Let
$S_X=\sqcup_{i=0}^2 C^X_i$ where each $C^X_i=\{\varepsilon_{i1},\varepsilon_{i2},\varepsilon_{i3}\}\subset \Aff^1_k$ has three points. Let $\psi:V\longrightarrow\PP^1_k$ be given by the Witt vector $(f_1,f_2)$, where $f_1$ (resp. $f_2$) has poles of order $2$ at $\varepsilon_{0j}$ (resp. $\varepsilon_{1j}$) for $1\le j\le 3$. 
Let $H=(\Z/11\Z)^{853}$ has a $G$-module decomposition $$H\cong\P0(V)[11]\oplus (\FF_{11})^3 \oplus \left(\frac{\FF_{11}[T]}{(T-3)}\right)^4
\oplus \left(\frac{\FF_{11}[T]}{(T-4)}\right)^3
\oplus \left(\frac{\FF_{11}[T]}{(T^5-9)}\right)^2
\oplus \left(\frac{\FF_{11}[T]}{(T^5-5)}\right)$$ as a $G$-module and $\theta$ be the corresponding $G$-action (where $\sigma$ acts on the non-trivial summands by multiplication by $T$). 

We want to calculate the minimum $g_{\min}$ of genera of $((\Z/11\Z)^{853}\rtimes_\theta \Z/25\Z)$-covers of $X$, \'etale away from $S_X$, dominating $V$. 

Note that by Subsection~\ref{ASW_eqn}, $k(V)=k(x)[x_1,x_2]/I$, where $$I=\langle x_1^5-x_1-f_1, x_2^5-x_2+x_1^{21}-2x_1^{17}+2x_1^{13}-x_1^9-f_2\rangle.$$ Note that $r_0=r_1=r_2=3$ and $r_X=\sum_i r_i=9$. Let $X_1\longrightarrow \PP^1_k$ be the ($\Z/5\Z$-Galois) intermediate cover given by the normalization of $\PP^1_k$ in $k(x,x_1)$ ($\subset k(V)$). It is totally ramified over $C^X_0$ as $f_1$ has poles of order $2$ only at those points. Since $x_1\in k(X_1)$ has poles (of order $2$) only at those ramified points and $f_2$ does not, $V\longrightarrow \PP^1_k$ is totally ramified at points over $C^X_0$. Since $f_2-(x_1^{21}-2x_1^{17}+2x_1^{13}-x_1^9)$ has poles (of order coprime to $5$) only at points in $X_1$ over $C^X_0\sqcup C^X_1$, $V\longrightarrow X_1$ is totally ramified at these points, i.e., $V\longrightarrow \PP^1_k$ is ramified at points in $C^X_1$, with ramification index $5$. In fact, the order of the pole of $f_2-(x_1^{21}-2x_1^{17}+2x_1^{13}-x_1^9)$ at the unique point in $X_1$ over $x_{0,j}$ (respectively, at each of the 5 points in $X_1$ over $x_{1,j}$) is $42$ (respectively $2$),   Let $S_V=\psi^{-1}(S_X)$. Now we are in the setup of the above corollary.

Since $11^1\equiv 1$ modulo $5$, $11^2\equiv 1$ modulo $25$, we have $d_1=1$, $d_2=5$. Also, $3^2=9$, $3^3=5$, and $3^4=4$ in $\FF_{11}$. Then we have a $G$-module decomposition: $$P_{1,1}(V\setminus S_V)\cong \P0(V)[11]\oplus (\FF_{11})^8\oplus \left(\oplus_{t=1}^4 \left(\frac{\FF_{11}[T]}{(T-3^t)}\right)^6 \right)\oplus \left(\oplus_{t=1}^4 \left(\frac{\FF_{11}[T]}{(T^5-3^t)}\right)^3 \right).$$ We have $H=H'\oplus H''$, where $H'=\P0(V)[11]$. We want to define an injective homomorphism $H''\hookrightarrow P_{\ram}(S_V,l)$. Here $\max_t \alpha_{2,t}=2,$, $\max_t \alpha_{1,t}=4$, $\alpha_0+1=4$.

By Theorem~\ref{min_alg}, we have to minimize
$y_0+5y_1+25y_2$, subject to the following conditions: $y_2\ge 2, y_1+y_2\ge 4, y_0+y_1+y_2\ge 4, 0\le y_s\le 3, y_s\in\Z$.

Clearly the minimum value is attained at $y_2=2$, $y_1=2$, $y_0=0$ and the minimum value is $c_{\min}=0+5\times 2+25\times 2=60$.

Again, by the Riemann-Hurwitz formula and the Hilbert different formula (see Subsection~\ref{genus-AS}), we obtain the genus.
If $g_1$ is the genus of $X_1$, then $2g_1=2+5(0-2)+3\times (2+1)(5-1)=28$ and hence $$g_V=1+\frac{5(28-2)+3\times(42+1)(5-1)+3\times5\times(2+1)(5-1)}{2}=414.$$
Note that $|H|=11^{828+3+(4+3)\times 1+(2+1)\times 5}=11^{853}$. Also $828=2g_V=|\P0(V)[11]|$.
Then the minimal genus is given by $$g_{\min}=\frac{2+11^{853}(828-2)+11^{853-1}(11-1)(c_{\min})}{2}=1+11^{852}\times 4843.$$
\end{ex}

\begin{rmk}
 In Theorem~\ref{min_gen}, when $S_X=C^X_0\sqcup C^X_1$ we can still compute the minimum genus $g_{\min}=\left[2+|H|(2g_V-2)+l^{n-1}(l-1)\mathfrak{c}\right]/{2}$. Let $r_i=|C^X_i|$. Then $\mathfrak{c}$ is as follows. 
    \begin{enumerate}
    \item If $\alpha_{1,t}=0$ for all $t\ge 1$
	\begin{enumerate}
	\item if $\alpha_0\leq r_0-1$, then $\mathfrak{c}=(\alpha_0+1)$,
	\item if $r_0\leq \alpha_0\leq r_0+r_1-1$, then $\mathfrak{c}=[(\alpha_0+1-r_0)p+r_0]$. 
	\end{enumerate}
    \item If $\alpha_{1,t}\neq 0$ for some $t\ge 1$
	\begin{enumerate}
	\item if $\alpha_0+1\leq \max_{t>0}\{\alpha_{1,t}\}$, then $\mathfrak{c}=(\max_{t>0}\{\alpha_{1,t}\})p$,
	\item if $\max_{t>0}\{\alpha_{1,t}\}<\alpha_0+1\leq \max_{t>0}\{\alpha_{1,t}\}+r_0$, then $$\mathfrak{c}=[(\max_{t>0}\{\alpha_{1,t}\})p+(\alpha_0+1-\max_{t>0}\{\alpha_{1,t}\})],$$
	\item if $\max_{t>0}\{\alpha_{1,t}\}+r_0<\alpha_0+1\leq r_0+r_1 $, then $\mathfrak{c}=[(\alpha_0+1-r_0)p+r_0]$.
	\end{enumerate}
\end{enumerate}

	For proof see \cite[Corollary~5.21]{pThesis}.
\end{rmk}

\begin{rmk}
Theorems \ref{min_alg} and \ref{min_gen} and the above remark can also be phrased in terms of finding the minimum genus among all covers providing proper solutions for the ``embedding problem" $(\beta:H\rtimes_\theta G\onto G, \alpha:\pi_1^{\e}(X\setminus S_X)\onto G)$ where the map $\alpha$ corresponds to the $G$-cover $\psi: V\longrightarrow X$.
\end{rmk}

\section{The case of covers of $\PP^1$ branched only at infinity}\label{last_section}
Now we specialize to the case when $X$ is $\PP^1$, $S_X=\{\infty\}$ and hence $X\setminus S_X$ is $\Aff^1$. Recall that for $1\le s\le a$, $d_s$ is the order of $l$ in $\Z/p^s\Z$. Note that the group action  $(\Z/p^a\Z)\longrightarrow\Aut((\Z/l\Z)^n)$ varies through out this section; $n$ varies only in Remark~\ref{rmk-p-cyclic} and Theorem~\ref{p^2-Galois}. The value of $a$ is fixed in this section. In  Proposition~\ref{p-cyclic} and Remark~\ref{rmk-p-cyclic} we fix $a=1$, whereas $a=2$ in Corollary~\ref{p2}, Theorem~\ref{p^2-Galois} and the last example.



\begin{lm}\label{quasi-p}
Let $a$ and $n$ be fixed positive integers. Then there exists a quasi-$p$ group $((Z/l\Z)^n\rtimes (\Z/p^a\Z))$ corresponding to some group action $\Z/p^a\Z\longrightarrow \Aut((Z/l\Z)^n)$ if and only if the order of $l$ in $(\Z/p^a\Z)^*$, $d_1$, divides $n$.
\end{lm}
\begin{proof}
Let the group $((Z/l\Z)^n\rtimes (\Z/p^a\Z))$ be quasi-$p$. Then it can not have a prime-to-$p$ quotient. Also $(Z/l\Z)^n$ is a $(\Z/p^a\Z)$-representation over $\FF_l$. Then we have the following decomposition into irreducible $(\Z/p^a\Z)$-submodules
$$(\Z/l\Z)^n=H_1\oplus \ldots \oplus H_t.$$
Then $((Z/l\Z)^n\rtimes (\Z/p^a\Z))$ is the fibre product of $ H_j\rtimes (\Z/p^a\Z)$ for $1\le j\le t$ over $\Z/p^a\Z$. If any $H_j$ is a $(\Z/p^a\Z)$-representation with trivial $(\Z/p^a\Z)$-action, then $H_j\rtimes (\Z/p^a\Z)$ is the direct product. Hence $((Z/l\Z)^n\rtimes (\Z/p^a\Z))$ will have a prime-to-$p$ quotient group isomorphic to $H_j$ (via the composition of maps $(Z/l\Z)^n\rtimes (\Z/p^a\Z)\onto H_j\times (\Z/p^a\Z) \onto H_j$). This is a contradiction. Thus, all the $H_j$'s are non-trivial irreducible representations. Then each will have dimension $d_s$ over $\FF_l$, for some $s$, $1\le s\le a$. But $d_1\mid d_s$, and $n$ is the sum of these dimensions $d_s$. Then $d_1\mid n$. 

Conversely, if $d_1\mid n$, we can write $(Z/l\Z)^n$ as $\oplus_{n/d_1} ((\Z/l\Z)^{d_1})$, where $(\Z/l\Z)^{d_1}$ is a non-trivial irreducible $(\Z/p^a\Z)$-representation of dimension $d_1$ (see Notation~\ref{basic_not}). For this $\Z/p^a\Z$-action let $((Z/l\Z)^n\rtimes (\Z/p^a\Z))$ be the corresponding semidirect product and $Q$ be its quasi-$p$ subgroup. If $Q$ is a proper subgroup of $((Z/l\Z)^n\rtimes (\Z/p^a\Z))$ then $((Z/l\Z)^n\rtimes (\Z/p^a\Z))/Q \cong (\Z/l\Z)^m$ for some nonzero $m$. Let $H'=(Z/l\Z)^n\cap Q$. Since $(Z/l\Z)^n$ and $Q$ are normal subgroups of $((Z/l\Z)^n\rtimes (\Z/p^a\Z))$, so is $H'$. Hence, $H'$ is a $(\Z/p^a\Z)$-subrepresentation of $(Z/l\Z)^{n}$ of dimension $n-m$. Hence there exist a $(\Z/p^a\Z)$-subrepresentation $H''$ of $(Z/l\Z)^{n}$ with $H'\oplus H''=(\Z/l\Z)^n$. For $h\in H''$ and $g\in \Z/p^a\Z$, note that $(ghg^{-1})h^{-1}$ is in $H''$ and in quasi-$p$ subgroup $Q$ (as $(ghg^{-1})h^{-1}=g(hg^{-1}h^{-1})$). But $Q\cap H''$ is the trivial subgroup, hence $ghg^{-1}=h$, i.e. $H''$ is a $\Z/p^a\Z$-representation with trivial $\Z/p^a\Z$-action. Contradiction to our choice of $(\Z/p^a\Z)$-action on $(\Z/l\Z)^n$.
\end{proof}

\begin{rmk}
Note that the lemma also follows from Proposition~\ref{n_char} and Lemma~\ref {genus_subcover}. But we have given a more elementary proof suggested to us by the referee.
\end{rmk}

\begin{notation}
For any real number $\varepsilon$, $\lceil \varepsilon\rceil$ denotes the lowest integer greater than or equal to $\varepsilon$, and $\lceil \varepsilon\rceil'$ denotes the lowest integer, not divisible by $p$, greater than or equal to $\varepsilon$. 
\end{notation}
\begin{pro}\label{p-cyclic}
Let $a=1$ and $n$ be a fixed positive integer divisible by $d_1$ and let $\delta'_1=\lceil\frac{n}{p-1}+1\rceil'$. The minimum genus of all $((\Z/l\Z)^n\rtimes (\Z/p\Z))$-Galois covers of $\PP^1$, \'etale away from $\infty$, for a fixed $n$ and all homomorphisms $\Z/p\Z\longrightarrow \Aut((\Z/l\Z)^n)$ is $$1+\frac{l^n}{2}(p\delta'_1-\delta'_1-p-1).$$
\end{pro}

\begin{proof}
Let $\psi:V\longrightarrow \PP^1$ be a $\Z/p\Z$-Galois cover, \'etale over $\Aff^1=\Spec k[x]$. By Artin-Schreier theory, $k(V)$ is given by $k(V)=k(x)[y]/(y^p-y-f_1(x))$ for some polynomial $f_1$ of degree $\delta_1$ such that $p\nmid \delta_1$. Moreover, $V\longrightarrow \PP^1$ is totally ramified above $\infty$ and the genus $g_V$ of $V$ is $(p-1)(\delta_1-1)/2$ (see Subsection~\ref{genus-u}).

Let $\gamma_1\le {2g_V}/{d_1}$ and $n=\gamma_1d_1$. Then there exists a $((\Z/l\Z)^n\rtimes (\Z/p\Z))$-Galois cover $W\longrightarrow\PP^1$, \'etale over $\Aff^1$, dominating $V\longrightarrow X$, by Lemma~\ref{genus_subcover} and Proposition~\ref{n_char}.
Then $W\longrightarrow V$ is a $(\Z/l\Z)^n$-cover possibly branched at $\psi^{-1}(\infty)$. But $\psi^{-1}(\infty)$ is singleton. Since any $(\Z/l\Z)^n$-cover is either \'etale or branched at two or more points, hence $W\longrightarrow V$ is an \'etale cover. So by Riemann-Hurwitz formula (see Subsection~\ref{genus-tame}), its genus is
$$g_W=\frac{2+l^n((p-1)(\delta_1-1)-2)}{2}.$$

The minimum of $g_W$ is attained at the minimum value of $\delta_1$. Since $$n=\gamma_1d_1\le 2g_V=(p-1)(\delta_1-1),$$ 
we have $1+n/(p-1)\le \delta_1$. Hence the minimum value of $\delta_1$ is $\lceil 1+n/(p-1)\rceil'=\delta'_1$. Then the minimum value of $g_W$ is $1+l^n(p\delta'_1-\delta'_1-p-1)/2$.
\end{proof}

\begin{rmk}\label{rmk-p-cyclic}
Here, $a=1$ is fixed, but $n>0$ varies. Since the value of $\delta'_1$ is $\lceil 1+n/(p-1)\rceil$ (if not divisible by $p$) or $\lceil 1+n/(p-1)\rceil+1$, the minimum genus above is an increasing function on $n$. If $n$ also varies, the minimum genus of all $((\Z/l\Z)^n\rtimes (\Z/p\Z))$-Galois cover of $\PP^1$, \'etale away from $\infty$, for any positive integer $n$ and any homomorphism $\Z/p\Z\longrightarrow \Aut((\Z/l\Z)^n)$ is $1+l^{d_1}(p-3)/2$ if $p$ is odd. When $p=2$, then $d_1=1$ and the minimum genus is $1$. This recovers parts 1 and 2 of \cite[Theorems 4.2 and 4.3]{G}.
\end{rmk}

In Subsection~\ref{genus-u}, we discussed about the upper jumps $u_i, ~\ 1\le i\le a,$ of a $p^a$-cyclic cover $\psi:V\longrightarrow \PP^1$, \'etale over $\Aff^1$. Let $X_i\longrightarrow \PP^1$ be the $p^i$-cyclic quotient cover of $\psi$, with genus $g_i$, $0\le i\le a$. Here $g_0=0$, $g_a=g_V$ for $V=X_a$. Recall that the upper jumps satisfy Condition~\ref{uJump}.
Also, given any $(u_1, \ldots ,u_a)$ satisfying the Condition~\ref{uJump}, there are $p^a$-cyclic covers of $\PP^1$ \'etale over $\Aff^1$ whose $i$-th upper jump at $\infty$ is $u_i$, $1\le i\le a$ (see \cite[Theorem~8.31]{Jac}).

From Equation \eqref{G} the genus of $X_i$ is $g_i=[1-p^i+(p-1)\sum_{t=1}^ip^{t-1}u_t]/2$. 

\begin{thm}\label{p^a}
Let $a$ and $n$ be fixed integers such that $d_1$ divides $n$. Let 
$$\mathfrak{U}=\{\underline{\gamma}=(\gamma_1,\ldots,\gamma_a)\in\Z^a_{\ge 0}\mid n=\sum_{i=1}^a\gamma_id_i\}$$ and $$c_{\min}=\min_{\underline{\gamma}\in\mathfrak{U}}\{
\sum_{i=1}^ap^{i-1}u_i(\underline{\gamma})\},$$ where
\begin{itemize}
 \item $u_1(\underline{\gamma})$ is the minimum integer such that $u_1(\underline{\gamma})\ge \frac{\gamma_1d_1}{p-1}+1$ and $\gcd(u_1(\underline{\gamma}),p)=1$;
 \item For $i\ge 2$, $u_i(\underline{\gamma})$ is the minimum integer such that $$u_i(\underline{\gamma})\ge \max\left\{\frac{\gamma_id_i}{p^{i-1}(p-1)}+1, ~\ pu_{i-1}(\underline{\gamma})\right\}; \text{ and } u_i(\underline{\gamma})=pu_{i-1}(\underline{\gamma})\text{ or } p\nmid u_i(\underline{\gamma}).$$
\end{itemize}
 Then the minimum genus of all $((\Z/l\Z)^n\rtimes (\Z/p^a\Z))$-Galois cover of $\PP^1$, \'etale away from $\infty$, for a fixed $n$ divisible by $d_1$ and all homomorphisms $\Z/p^a\Z\longrightarrow \Aut((\Z/l\Z)^n)$ is
$$1+l^n\left(\frac{-1-p^a}{2}+\frac{(p-1)c_{\min}}{2}\right).$$
\end{thm}

\begin{proof}
By Lemma~\ref{quasi-p} and Abhyankar's Conjecture (Raynaud's theorem \cite{R}) on the affine line, there exists a $((\Z/l\Z)^n\rtimes_\theta (\Z/p^a\Z))$-Galois cover $W\longrightarrow \PP^1$, \'etale away from $\infty$, for some $\theta:\Z/p^a\Z\longrightarrow \Aut((\Z/l\Z)^n)$. If $V\longrightarrow \PP^1$ is the corresponding $p^a$-cyclic cover dominated by $W$, then by  Proposition~\ref{n_char} and Lemma~\ref{genus_subcover}, $n$ can be expressed as $n=\sum_{i=1}^a \gamma_id_i$ for some $\underline{\gamma}\in \mathfrak{U}$ such that $\gamma_i\le 2(g_i-g_{i-1})/{d_i}$ where $g_0=0$, $g_a=g_V$.

Let $(u_1, \ldots, u_a)$ be the upper jumps at $\infty$. Then it satisfies Condition~\ref{uJump}. Then for $i=1$, we see that $\gamma_1d_1\le 2g_1=1-p+(p-1)u_1$ (by Equation \eqref{G}), $u_1\ge 1$ and $\gcd(u_1,p)\neq 1$. Thus
$$u_1\ge \frac{\gamma_1d_1}{p-1}+1 \text{ and } \gcd(u_1,p)\neq 1.$$
Hence $u_1\ge u_1(\underline{\gamma})$.
Similarly, for $i\ge 2$, $u_i\ge\max\{\gamma_id_i/(p^{i-1}(p-1))+1,pu_{i-1}\}$, $u_i-pu_{i-1}$ is zero or a positive integer coprime to $p$. Note that $g_W=1+l^n(g_V-1)$ (see Subsection~\ref{genus-tame}), and $g_V=(1-p^i+(p-1)c)/{2}$ for $c=\sum_{i=1}^au_ip^{i-1}$. Therefore, to minimize $g_W$, we need to minimize $g_V$. This happens when $c$ reaches the minimum. But $c$ is an increasing function of each $u_i$. Note that each $u_i(\underline{\gamma})$ is also an increasing function of $u_{i-1}(\underline{\gamma})$ for $i\ge 2$. Indeed, if $ \gamma_id_i/(p^{i-1}(p-1))+1\le pu_{i-1}(\underline{\gamma})$ and $u_{i-1}(\underline{\gamma})$ increases by $1$, then $u_i(\underline{\gamma})$ will increase by $p$. So, for each weighted partition $\underline{\gamma}$ of $n$, $c$ attains the minimum value at $(u_1, \ldots, u_a)=(u_1(\underline{\gamma}), \ldots, u_a(\underline{\gamma}))$.
\end{proof}

\begin{co}\label{p2}
Let $a=2$, $d_1$ divides $n$ and $n/d_1=R+S\times (d_2/d_1)$ for non-negative integers $R$ and $S$, $0\le R<d_2/d_1$. The minimum genus of all $((\Z/l\Z)^n\rtimes (\Z/p^2\Z))$-Galois cover of $\PP^1$, \'etale away from $\infty$, for a fixed $n$ divisible by $d_1$ and all group actions $\Z/p^2\Z\longrightarrow \Aut((\Z/l\Z)^n)$ is
$$1+l^n\left(\frac{-1-p^2+(p-1)\min\{u_1(j)+pu_2(j): 0\le j\le S\}}{2}\right),$$ where $u_1(j)$ is the minimum integer such that $u_1(j)\ge(Rd_1+jd_2)/(p-1)+1$ and $\gcd(u_1(j),p)=1$, $u_2(j)$ is the minimum integer such that $u_2(j)\ge\max\{\frac{Sd_2-jd_2}{p(p-1)}+1,pu_1(j)\}$, and either $u_2(j)=pu_1(j)$ or $p\nmid u_2(j)$.
\end{co}

\begin{proof}
Here $a=2$ and $n=Rd_1+Sd_2=(R+jd_2/d_1)d_1+(S-j)d_2$. Hence we have $\mathfrak{U}=\{(R+jd_2/d_1,S-j)\mid 0\le j\le S\}$. We replace $u_i((R+jd_2/d_1,S-j))$ in Proposition~\ref{p^a} by $u_i(j)$, $i=1,2$. Now the result follows from Proposition~\ref{p^a}.
\end{proof}

Let $a=2$ be fixed, but $n>0$ be varying. We now proceed to prove Theorem~\ref{p^2-Galois}, which provides a more precise result concerning the minimum genus of all $((\Z/l\Z)^n\rtimes(\Z/p^2\Z))$-Galois covers of $\PP^1$, \'etale over $\Aff^1$ (where $n$ and the group action are varying).

\begin{proof}[Proof of Theorem~\ref{p^2-Galois}]
 Again note that $((\Z/l\Z)^n\rtimes(\Z/p^a\Z))$-cover of $\PP^1$, \'etale over $\Aff^1$, exist iff $d_1$ divides $n$. We use Corollary~\ref{p2} with its notation.

Let us first assume that $n=d_1=d_2$. Then $R=0$ and $S=1$. Since $u_1(0)\ge 1$ and $\gcd(u_1(0),p)=1$, $u_1(0)=1$. Since $d_2=d_1 \le p-1$ and $u_2(0)$ is the smallest integer such that $u_2(0)\ge \max\{d_2/(p(p-1))+1,p\}=p$, $u_2(0)=p$. Then $u_1(0)+pu_2(0)=1+p^2$.

If $p=2$, $u_1(1)=3$, $u_2(1)\ge\max\{(0\times d_2)/(2(2-1))+1,3\times 2\}$. So $u_2(1)$ must be $6$ and hence $u_1(1)+pu_2(1)=15>5=p^2+1$ for $p=2$. Similarly, $u_1(1)+pu_2(1)=4+4p^2 >p^2+1$ for $p\neq 2$.

When $n>d_1=d_2$, $R=0$, $S=n/d_1\ge 2$ and $0\le j\le S$. Then $u_1(0)=1$, $u_2(0)\ge p$; for $j>0$, $u_1(j)\ge 2$, $u_2(j)\ge 2p$. So $u_1(j)+pu_2(j)\ge 1+p^2$. Hence, when $d_1=d_2$, the minimum genus is achieved when $n=d_1=d_2$ and in this case $c$ in the statement of Theorem~\ref{p^2-Galois} is $\min\{u_1(j)+pu_2(j): 0\le j\le 1\}=1+p^2$ by the above corollary.

Now consider the case $d_1<d_2$. If $n=d_1$ then $R=1$ and $S=0$. Since $j=0$, $u_1(0)\ge d_1/(p-1)+1$ and $\gcd(u_1(0),p)=1$, $u_1(0)$ is $3$ (resp. $2$) if $p=2$ (resp. $p \neq 2$). Similarly, $u_2(0)$ is $6$ (resp. $2p$) if $p=2$ (resp. $p\neq 2$). Then $u_1(0)+pu_2(0)$ is $15$ (resp. $2+2p^2$) if $p=2$ (resp. $p\neq 2$). Hence, the minimum genus when $n=d_1$ is obtained by putting $c=15$ if $p=2$ and $c=2+2p^2$ if $p$ is an odd prime.

When $d_1<n<d_2$ then $R=n/d_1\ge 2$ and $S=0$. Then $u_1(0)+pu_2(0)$ is same as $n=d_1<d_2$ case, but because of the $l^n$ factor, the genus will be bigger than $n=d_1<d_2$ case.

When $n\ge d_2>d_1$ then $S\ge 1$. Then similarly, we see that the genus in this case is bigger than the minimum genus for the $n=d_1<d_2$ case.
\end{proof}

\begin{ex}
Let $p=3$, $l=2$, $a=2$ and $n=14$. We will show that there exist $((\Z/2\Z)^{14}\rtimes (\Z/9\Z))$-Galois covers of $\PP^1$, \'etale over $\Aff^1$, for some $\Z/9\Z\rightarrow \Aut((\Z/2\Z)^{14})$ and the minimum of genera of all such covers (where the group action varies) is  $245761$.

Here $d_1=2$, $d_2=6$, $n=14$. By Lemma~\ref{quasi-p} and Abhyankar's Conjecture on the Affine line, such covers exist. Note that $14/2=1+2\times(6/2)$ and so $R=1$, $S=2$ and $j=0,1,2$. 
\begin{center}
\begin{tabular}{ c|c|c|c } 

 $j$ & $u_1(j)$ & $u_2(j)$ & $u_1(j)+3u_2(j)$ \\ \hline
 0 & 2 & 6 & 20 \\ 
 1 & 5 & 15 & 47 \\  
 2 & 8 & 24 & 80 \\ \end{tabular}
\end{center}
Since the minimum of $\delta_1(j)+3\delta_2(j)$ is $24$, by the above corollary, the required minimum genus is $$1+2^{14}\left(\frac{-1-3^2+(3-1)\times 20}{2}\right) =1+2^{13}\times 30=245761.$$

\end{ex}

\end{document}